\documentclass[11pt]{amsart}
\usepackage{amssymb,amsmath,amsfonts,latexsym,amsthm,geometry}
\usepackage{amsmath}
\usepackage{amsfonts}
\usepackage{amssymb}
\usepackage{graphicx}
\usepackage{color}
\usepackage[colorlinks=true,linkcolor=blue,citecolor=blue]{hyperref}

\providecommand{\U}[1]{\protect\rule{.1in}{.1in}}
\providecommand{\U}[1]{\protect\rule{.1in}{.1in}}
\providecommand{\U}[1]{\protect\rule{.1in}{.1in}}
\newtheorem{theorem}{Theorem}[section]
\newtheorem{corollary}[theorem]{Corollary}
\newtheorem{proposition}[theorem]{Proposition}
\newtheorem{lemma}[theorem]{Lemma}

\newtheorem*{teoremavectorial}{{Theorem}}
\theoremstyle{definition}

\newtheorem{example}[theorem]{Example}

\newtheorem{remark}[theorem]{Remark}

\begin{document}
\title[Optimal exponents for Hardy--Littlewood inequalities for $m$-linear
operators]{Optimal exponents for Hardy--Littlewood inequalities for $m$%
-linear operators}
\author[Aron]{R. M. Aron}
\address{Department of Mathematical Sciences\\
\indent Kent State University\\
\indent Kent, Ohio 44242, USA}
\email{aron@math.kent.edu}
\author[N\'u\~nez]{D. N\'{u}\~{n}ez-Alarc\'{o}n}
\address{Departamento de Matem\'{a}tica\\
\indent Universidade Federal de Pernambuco\\
\indent50.740-560 - Recife, Brazil\\
\indent\& Department of Mathematical Sciences\\
\indent Kent State University\\
\indent Kent, Ohio 44242, USA}
\email{danielnunezal@gmail.com}
\author[Pellegrino]{D. M. Pellegrino}
\address{Departamento de Matem\'{a}tica\\
\indent Universidade Federal da Para\'{\i}ba,\\
\indent58.051-900 - Jo\~{a}o Pessoa, Brazil}
\email{pellegrino@pq.cnpq.br and dmpellegrino@gmail.com}
\author[Serrano]{D. M. Serrano-Rodr\'{\i}guez}
\address{Departamento de Matem\'{a}tica\\
\indent Universidade Federal de Pernambuco\\
\indent50.740-560 - Recife, Brazil\\
\indent\& Department of Mathematical Sciences\\
\indent Kent State University\\
\indent Kent, Ohio 44242, USA}
\email{dmserrano0@gmail.com}
\thanks{R. Aron is supported in part by MINECO MTM2014-57838-C2-2-P and
Prometeo II/2013/013, D. N\'{u}\~{n}ez is supported by Capes Grant
000785/2015-06, D. Pellegrino is supported by CNPq and D. M. Serrano is
supported by CAPES Grant 000786/2015-02.}

\begin{abstract}
The Hardy--Littlewood inequalities on $\ell _{p}$ spaces provide optimal
exponents for some classes of inequalities for bilinear forms on $\ell _{p}$
spaces. In this paper we investigate in detail the exponents involved in
Hardy--Littlewood type inequalities and provide several optimal results that
were not achieved by the previous approaches. Our first main result asserts
that for $q_{1},...,q_{m}>0$ and an infinite-dimensional Banach space $Y$
attaining its cotype $\cot Y$, if
\begin{equation*}
\frac{1}{p_{1}}+...+\frac{1}{p_{m}}<\frac{1}{\cot Y},
\end{equation*}%
then the following assertions are equivalent:

(a) There is a constant $C_{p_{1},...,p_{m}}^{Y}\geq 1$ such that%
\begin{equation*}
\left( \sum_{j_{1}=1}^{\infty }\left( \sum_{j_{2}=1}^{\infty }\cdots \left(
\sum_{j_{m}=1}^{\infty }\left\Vert A(e_{j_{1}},...,e_{j_{m}})\right\Vert
^{q_{m}}\right) ^{\frac{q_{m-1}}{q_{m}}}\cdots \right) ^{\frac{q_{1}}{q_{2}}%
}\right) ^{\frac{1}{q_{1}}}\leq C_{p_{1},...,p_{m}}^{Y}\left\Vert
A\right\Vert
\end{equation*}%
for all continuous $m-$linear operators $A:\ell _{p_{1}}\times \cdots \times
\ell _{p_{m}}\rightarrow Y.$

(b) The exponents $q_{1},...,q_{m}$ satisfy%
\begin{equation*}
q_{1}\geq \lambda _{m,\cot Y}^{p_{1},...,p_{m}},q_{2}\geq \lambda _{m-1,\cot
Y}^{p_{2},...,p_{m}},...,q_{m}\geq \lambda _{1,\cot Y}^{p_{m}},
\end{equation*}%
where, for $k=1,...,m,$
\begin{equation*}
\lambda _{m-k+1,\cot Y}^{p_{k},...,p_{m}}:=\frac{\cot Y}{1-\left( \frac{1}{%
p_{k}}+...+\frac{1}{p_{m}}\right) \cot Y}.
\end{equation*}%
As an application of the above result we generalize to the $m$-linear
setting one of the classical Hardy--Littlewood inequalities for bilinear
forms. Our result is sharp in a very strong sense: the constants and
exponents are optimal, even if we consider mixed sums.
\end{abstract}

\maketitle

\section{Introduction}

Let $\mathbb{K}$ be the real or complex scalar field. In 1934 Hardy and
Littlewood proved three theorems (Theorems \ref{hl1}, \ref{hl2}, \ref{hl3},
below) on the summability of bilinear forms on $\ell _{p}\times \ell _{q}$
(here, and henceforth, when $p=\infty $ we consider $c_{0}$ instead of $\ell
_{\infty })$. For any function $f$ we shall consider $f(\infty
):=\lim_{s\rightarrow \infty }f(s)$ and for any $s\geq 1$ we denote the
conjugate index of $s$ by $s^{\ast },$ i.e., $\frac{1}{s}+\frac{1}{s^{\ast }}%
=1$.

For all $p,q\in (1,\infty ]$, such that $\frac{1}{p}+\frac{1}{q}<1$, let us
define
\begin{equation*}
\lambda :=\frac{pq}{pq-p-q},
\end{equation*}%
and
\begin{equation*}
\mu =\frac{4pq}{3pq-2p-2q}.
\end{equation*}

If $p$ and $q$ are simultaneously $\infty $, then $\lambda $ and $\mu $ are $%
1$ and $4/3$ respectively.

\smallskip

From now on, $\left( e_{k}\right) _{k=1}^{\infty }$ denotes the sequence of
canonical vectors in $\ell _{p}$.

\begin{theorem}
\label{hl1}(See Hardy and Littlewood \cite[Theorem 1]{hardy}) Let $p,q\in
\lbrack 2,\infty ]$, with $\frac{1}{p}+\frac{1}{q}\leq \frac{1}{2}$. There
is a constant $C_{p,q}\geq 1$ such that%
\begin{equation}
\left( \sum_{j_{1}=1}^{\infty }\left( \sum_{j_{2}=1}^{\infty }\left\vert
A(e_{j_{1}},e_{j_{2}})\right\vert ^{2}\right) ^{\frac{\lambda }{2}}\right) ^{%
\frac{1}{\lambda }}\leq C_{p,q}\left\Vert A\right\Vert ,  \label{hl1a}
\end{equation}%
and%
\begin{equation}
\left( \sum_{j_{1},j_{2}=1}^{\infty }\left\vert
A(e_{j_{1}},e_{j_{2}})\right\vert ^{\mu }\right) ^{\frac{1}{\mu }}\leq
C_{p,q}\left\Vert A\right\Vert ,  \label{hl1c}
\end{equation}%
for all continuous bilinear forms $A:\ell _{p}\times \ell _{q}\rightarrow
\mathbb{K}$.
\end{theorem}

It is well known that the exponents $\lambda $ and $\mu $ are optimal. Also,
in (\ref{hl1a}) the positions of the exponents $2$ and $\lambda $ can be
interchanged. Furthermore, $2$ and $\lambda $ can be replaced by $a,b\in
\lbrack \lambda ,2]$ provided that
\begin{equation*}
\frac{1}{a}+\frac{1}{b}\leq \frac{3}{2}-\left( \frac{1}{p}+\frac{1}{q}%
\right) .
\end{equation*}

\begin{theorem}
\label{hl2}(See Hardy and Littlewood \cite[Theorem 2]{hardy}) Let $p,q\in
\lbrack 2,\infty ]$, with $\frac{1}{2}<\frac{1}{p}+\frac{1}{q}<1$. There is
a constant $C_{p,q}\geq 1$ such that%
\begin{equation}
\left( \sum_{j_{1}=1}^{\infty }\left( \sum_{j_{2}=1}^{\infty }\left\vert
A(e_{j_{1}},e_{j_{2}})\right\vert ^{2}\right) ^{\frac{\lambda }{2}}\right) ^{%
\frac{1}{\lambda }}\leq C_{p,q}\left\Vert A\right\Vert ,  \label{hl2a}
\end{equation}%
and%
\begin{equation}
\left( \sum_{j_{1},j_{2}=1}^{\infty }\left\vert
A(e_{j_{1}},e_{j_{2}})\right\vert ^{\lambda }\right) ^{\frac{1}{\lambda }%
}\leq C_{p,q}\left\Vert A\right\Vert ,  \label{hl2c}
\end{equation}%
for all continuous bilinear forms $A:\ell _{p}\times \ell _{q}\rightarrow
\mathbb{K}$.
\end{theorem}

The exponent $\lambda $ above is also optimal. However, contrary to what
happens in Theorem \ref{hl1}, now, in (\ref{hl2a}) the exponents $2$ and $%
\lambda $ cannot be interchanged (see \cite{was}).

\begin{theorem}
\label{hl3}(See Hardy and Littlewood \cite[Theorem 3]{hardy}) Let $1<q<2<p$,
with $\frac{1}{p}+\frac{1}{q}<1$. There is a constant $C_{p,q}\geq 1$ such
that%
\begin{equation}
\left( \sum_{j_{1},j_{2}=1}^{\infty }\left\vert
A(e_{j_{1}},e_{j_{2}})\right\vert ^{\lambda }\right) ^{\frac{1}{\lambda }%
}\leq C_{p,q}\left\Vert A\right\Vert ,  \label{hl3c}
\end{equation}%
for all continuous bilinear forms $A:\ell _{p}\times \ell _{q}\rightarrow
\mathbb{K}$.
\end{theorem}

The \textquotedblleft optimal\textquotedblright\ exponent in (\ref{hl3c})
was improved in \cite{tonge}:

\begin{theorem}
\label{t1}(See Osikiewicz and Tonge \cite{tonge}) Let $1<q\leq 2<p$, with $%
\frac{1}{p}+\frac{1}{q}<1$. If $A:\ell _{p}\times \ell _{q}\rightarrow
\mathbb{K}$ is a continuous bilinear form, then
\begin{equation}
\left( \sum_{j_{1}=1}^{\infty }\left( \sum_{j_{2}=1}^{\infty }\left\vert
A(e_{j_{1}},e_{j_{2}})\right\vert ^{q^{\ast }}\right) ^{\frac{\lambda }{%
q^{\ast }}}\right) ^{\frac{1}{\lambda }}\leq \left\Vert A\right\Vert .
\label{t1b}
\end{equation}
\end{theorem}

Hardy--Littlewood type inequalities were extensively investigated in recent
years, but despite much progress there are still several open questions
concerning the optimality of exponents and constants.

\bigskip One of the main nuances on the optimality of exponents that
apparently has been overlooked in the past is that results of optimality of
exponents for expressions like%
\begin{equation*}
\left( \sum_{j_{1},...,j_{m}=1}^{\infty }\left\vert
A(e_{j_{1}},...e_{j_{m}})\right\vert ^{s}\right) ^{\frac{1}{s}}\leq
C\left\Vert A\right\Vert
\end{equation*}%
are in some sense sub-optimal. The main point is that the above inequality
can be viewed as%
\begin{equation}
\left( \sum_{j_{1}=1}^{\infty }...\left( \sum_{j_{m-1}=1}^{\infty }\left(
\sum_{j_{m}=1}^{\infty }\left\vert A(e_{j_{1}},...e_{j_{m}})\right\vert
^{s_{m}}\right) ^{\frac{1}{s_{m}}s_{m-1}}\right) ^{\frac{1}{s_{m-1}}%
}...\right) ^{\frac{1}{s_{1}}}\leq C\left\Vert A\right\Vert  \label{987}
\end{equation}%
for $s_{1}=...=s_{m}=s$, and this is the way that the optimality of the
exponents can be investigated with more accuracy. A simple illustration of
this fact is that the exponent $\lambda $ of (\ref{hl2c}) is optimal, but a
quick look at (\ref{hl2a}) shows that the optimality of (\ref{hl2c}) is just
apparent. An extensive investigation of the Hardy--Littlewood inequalities
in light of multiple sums like (\ref{987}) was initiated in \cite{abps2,
aaa, anss}, but there are still some subtle issues not encompassed by
previous work. One of the main technical obstacles is to develop methods to
find optimal exponents in situations in which the optimal exponents of each
sum cannot be interchanged. This is the case of our first main result (for
definition of cotype, see the next section):

\begin{teoremavectorial}
\bigskip \label{7788}(See Theorem \ref{661}, below) Let $q_{1},...,q_{m}>0$
and $Y$ be an infinite-dimensional Banach space attaining its cotype $\cot
Y. $ If
\begin{equation*}
\frac{1}{p_{1}}+...+\frac{1}{p_{m}}<\frac{1}{\cot Y},
\end{equation*}%
then the following assertions are equivalent:

(a) There is a constant $C_{p_{1},...,p_{m}}^{Y}\geq 1$ such that%
\begin{equation*}
\left( \sum_{j_{1}=1}^{\infty }\left( \sum_{j_{2}=1}^{\infty }\cdots \left(
\sum_{j_{m}=1}^{\infty }\left\Vert A(e_{j_{1}},...,e_{j_{m}})\right\Vert
^{q_{m}}\right) ^{\frac{q_{m-1}}{q_{m}}}\cdots \right) ^{\frac{q_{1}}{q_{2}}%
}\right) ^{\frac{1}{q_{1}}}\leq C_{p_{1},...,p_{m}}^{Y}\left\Vert
A\right\Vert
\end{equation*}%
for all continuous $m$-linear operators $A:\ell _{p_{1}}\times \cdots \times
\ell _{p_{m}}\rightarrow Y.$

(b) The exponents $q_{1},...,q_{m}$ satisfy
\begin{equation*}
q_{1}\geq \lambda _{m,\cot Y}^{p_{1},...,p_{m}},\text{ }q_{2}\geq \lambda
_{m-1,\cot Y}^{p_{2},...,p_{m}},...,\text{ }q_{m-1}\geq \lambda _{2,\cot
Y}^{p_{m-1},p_{m}},\text{ }q_{m}\geq \lambda _{1,\cot Y}^{p_{m}},
\end{equation*}%
where, for $k=1,...,m,$
\begin{equation*}
\lambda _{m-k+1,\cot Y}^{p_{k},...,p_{m}}:=\frac{\cot Y}{1-\left( \frac{1}{%
p_{k}}+...+\frac{1}{p_{m}}\right) \cot Y}.
\end{equation*}
\end{teoremavectorial}

Despite the wide generality of the results of \cite{abps2, aaa, dimant}, the
results of this paper do not follow from the techniques developed in these
earlier papers. We illustrate, by means of a concrete example, how the above
Theorem provides more precise information than previously known results.

\begin{example}
Suppose that $m=3,$ $p_{1}=p_{2}=p_{3}=10,$ and $Y=\ell _{3}.$ The above
Theorem implies that there is a universal constant $C\geq 1$ such that%
\begin{equation}
\left( \sum_{j_{1}=1}^{\infty }\left( \sum_{j_{2}=1}^{\infty }\left(
\sum_{j_{3}=1}^{\infty }\left\Vert
A(e_{j_{1}},e_{j_{2}},e_{j_{3}})\right\Vert ^{q_{3}}\right) ^{\frac{q_{2}}{%
q_{3}}}\right) ^{\frac{q_{1}}{q_{2}}}\right) ^{\frac{1}{q_{1}}}\leq
C\left\Vert A\right\Vert  \label{0099}
\end{equation}%
for all continuous $3$-linear forms $A:\ell _{10}\times \ell _{10}\times
\ell _{10}\rightarrow \ell _{3}$ if and only if%
\begin{equation*}
\left\{
\begin{array}{c}
q_{1}\geq 30, \\
q_{2}\geq \text{ }\frac{15}{2}, \\
q_{3}\geq \frac{30}{7},%
\end{array}%
\right.
\end{equation*}%
while the best previously known estimates (from \cite[Proposition 4.3]%
{dimant} and \cite[Theorem 1.5]{abps2}) just give that (\ref{0099}) is valid
for $q_{j}\geq 30$ for all $j=1,2,3$ and that we cannot have simultaneously $%
q_{1}=q_{2}=q_{3}<30.$
\end{example}

Our second main result, stated and proved in Section \ref{s55}, is an
application of this Theorem, generalizing Theorems \ref{hl3} and \ref{t1}
with optimal exponents, to the multilinear setting. In Section \ref{s4} we
show that the optimal constant for the scalar-valued case is precisely $1$,
and finally we remark how our results can be translated to the theory of
multiple summing operators.

\section{Optimal exponents: vector-valued case}

\bigskip Let $2\leq q<\infty $ and $0<s<\infty $. Recall that (see \cite%
{albiac}) a Banach space $X$ has \emph{cotype} $q$ if there is a constant $%
C>0$ such that, no matter how we select finitely many vectors $x_{1},\dots
,x_{n}\in X$,%
\begin{equation}
\left( \sum_{j=1}^{n}\Vert x_{j}\Vert ^{q}\right) ^{\frac{1}{q}}\leq C\left(
\int_{[0,1]}\left\Vert \sum_{j=1}^{n}r_{j}(t)x_{j}\right\Vert ^{s}dt\right)
^{1/s},  \label{99}
\end{equation}%
where $r_{j}$ denotes the $j$-th Rademacher function. It is well known that
if (\ref{99}) is satisfied for a certain $s>0$, then it is satisfied for all
$s>0.$ For a fixed $s$, the smallest of these constants will be denoted by $%
C_{q,s}(X)$ and the infimum of the cotypes of $X$ is denoted by $\cot X$. By
convention we denote $C_{q,2}(X)$ by $C_{q}(X)$.

\bigskip The following simple lemma will be useful.

\begin{lemma}
\label{propo}Let $Y$ be a Banach space, $m\geq 2$, $p_{1},...,p_{m}\in
\lbrack 1,\infty ]$, and $q_{1},...,q_{m},r_{2},...,r_{m}\in (0,\infty )$.
Assume that if
\begin{equation*}
\left( \sum_{j_{2}=1}^{\infty }\left( \sum_{j_{3}=1}^{\infty }\cdots \left(
\sum_{j_{m}=1}^{\infty }\left\Vert A(e_{j_{2}},...,e_{j_{m}})\right\Vert
^{q_{m}}\right) ^{\frac{q_{m-1}}{q_{m}}}\cdots \right) ^{\frac{q_{2}}{q_{3}}%
}\right) ^{\frac{1}{q_{2}}}<\infty
\end{equation*}%
for all continuous $\left( m-1\right) $-linear operators $A:\ell
_{p_{2}}\times \cdots \times \ell _{p_{m}}\rightarrow Y$, then $q_{i}\geq
r_{i}$ for all $i\in \left\{ 2,...,m\right\} .$ Then
\begin{equation*}
\left( \sum_{j_{1}=1}^{\infty }\left( \sum_{j_{2}=1}^{\infty }\cdots \left(
\sum_{j_{m}=1}^{\infty }\left\Vert B(e_{j_{1}},...,e_{j_{m}})\right\Vert
^{q_{m}}\right) ^{\frac{q_{m-1}}{q_{m}}}\cdots \right) ^{\frac{q_{1}}{q_{2}}%
}\right) ^{\frac{1}{q_{1}}}<\infty
\end{equation*}%
for all continuous $m$-linear operators $B:\ell _{p_{1}}\times \cdots \times
\ell _{p_{m}}\rightarrow Y$ implies that $q_{i}\geq r_{i}$ for all $i\in
\left\{ 2,...,m\right\} .$
\end{lemma}

\begin{proof}
~Let $A:\ell _{p_{2}}\times \cdots \times \ell _{p_{m}}\rightarrow Y$ be a
continuous $\left( m-1\right) $-linear operator and consider the continuous $%
m$-linear operator $B_{1}:\ell _{p_{1}}\times \cdots \times \ell
_{p_{m}}\rightarrow Y$ given by%
\begin{equation*}
B_{1}(x^{\left( 1\right) },...,x^{\left( m\right) })=x_{1}^{\left( 1\right)
}A\left( x^{\left( 2\right) },...,x^{\left( m\right) }\right) .
\end{equation*}%
Clearly $\left\Vert B_{1}(e_{1},e_{j_{2}},...,e_{j_{m}})\right\Vert
=\left\Vert A(e_{j_{2}},...,e_{j_{m}})\right\Vert $, and since
\begin{eqnarray*}
&&\left( \sum_{j_{1}=1}^{\infty }\left( \sum_{j_{2}=1}^{\infty }\cdots
\left( \sum_{j_{m}=1}^{\infty }\left\Vert
B_{1}(e_{j_{1}},...,e_{j_{m}})\right\Vert ^{q_{m}}\right) ^{\frac{q_{m-1}}{%
q_{m}}}\cdots \right) ^{\frac{q_{1}}{q_{2}}}\right) ^{\frac{1}{q_{1}}} \\
&=&\left( \sum_{j_{2}=1}^{\infty }\cdots \left( \sum_{j_{m}=1}^{\infty
}\left\Vert B_{1}(e_{1},e_{j_{2}},...,e_{j_{m}})\right\Vert ^{q_{m}}\right)
^{\frac{q_{m-1}}{q_{m}}}\cdots \right) ^{\frac{1}{q_{2}}} \\
&=&\left( \sum_{j_{2}=1}^{\infty }\cdots \left( \sum_{j_{m}=1}^{\infty
}\left\Vert A(e_{j_{2}},...,e_{j_{m}})\right\Vert ^{q_{m}}\right) ^{\frac{%
q_{m-1}}{q_{m}}}\cdots \right) ^{\frac{1}{q_{2}}},
\end{eqnarray*}%
the proof is done.
\end{proof}

\bigskip From now on, let $r\geq 2,$ and let $p_{1},...,p_{m}\,\in (r,\infty
]$ be such that
\begin{equation*}
\frac{1}{p_{1}}+...+\frac{1}{p_{m}}<\frac{1}{r}.
\end{equation*}%
For $k=1,...,m$, we define
\begin{equation*}
\lambda _{m-k+1,r}^{p_{k},...,p_{m}}:=\frac{r}{1-\left( \frac{1}{p_{k}}+...+%
\frac{1}{p_{m}}\right) r}.
\end{equation*}%
For a Banach space $Y$ and $1\leq s\leq \infty $, let $\ell _{s}\left(
Y\right) $ be the Banach space of $Y-$valued sequences $\left( y_{i}\right)
_{i=1}^{\infty }$ with the norm
\begin{equation*}
\left\Vert \left( y_{i}\right) _{i=1}^{\infty }\right\Vert _{\ell _{r}\left(
Y\right) }=\left( \sum_{i=1}^{\infty }\left\Vert y_{i}\right\Vert
_{Y}^{s}\right) ^{\frac{1}{s}},
\end{equation*}%
(the usual modification is required if $s=\infty $). When there is no
ambiguity, for a vector $y\in Y$, we denote the norm $\left\Vert
y\right\Vert _{Y}$ by $\left\Vert y\right\Vert $.

We now state and prove our first main theorem. As we mentioned before, it
improves \cite[Proposition 4.3]{dimant} and \cite[Theorem 1.5]{abps2}, by
providing the exact optimal exponents.

\begin{theorem}
\label{661}\bigskip\ Let $q_{1},...,q_{m}>0$, and $Y$ be an
infinite-dimensional Banach space with cotype $\cot Y$. If
\begin{equation*}
\frac{1}{p_{1}}+...+\frac{1}{p_{m}}<\frac{1}{\cot Y},
\end{equation*}%
then the following assertions are equivalent:

(a) There is a constant $C_{p_{1},...,p_{m}}^{Y}\geq 1$ such that%
\begin{equation*}
\left( \sum_{j_{1}=1}^{\infty }\left( \sum_{j_{2}=1}^{\infty }\cdots \left(
\sum_{j_{m}=1}^{\infty }\left\Vert A(e_{j_{1}},...,e_{j_{m}})\right\Vert
^{q_{m}}\right) ^{\frac{q_{m-1}}{q_{m}}}\cdots \right) ^{\frac{q_{1}}{q_{2}}%
}\right) ^{\frac{1}{q_{1}}}\leq C_{p_{1},...,p_{m}}^{Y}\left\Vert
A\right\Vert
\end{equation*}%
for all continuous $m$-linear operators $A:\ell _{p_{1}}\times \cdots \times
\ell _{p_{m}}\rightarrow Y.$

(b) The exponents $q_{1},...,q_{m}$ satisfy
\begin{equation*}
q_{1}\geq \lambda _{m,\cot Y}^{p_{1},...,p_{m}},q_{2}\geq \lambda _{m-1,\cot
Y}^{p_{2},...,p_{m}},...,q_{m-1}\geq \lambda _{2,\cot
Y}^{p_{m-1},p_{m}},q_{m}\geq \lambda _{1,\cot Y}^{p_{m}}.
\end{equation*}
\end{theorem}

\begin{proof}
From now on, we shall denote $r=\cot Y.$ The proof of the case $m=1$ can be
verified by using a short argument from the theory of absolutely summing
operators, but we prefer to present a self contained argument. It suffices
to note that
\begin{equation*}
\lambda _{1,\cot Y}^{p_{1}}=\frac{\cot Y}{1-\frac{\cot Y}{p_{1}}}=\frac{%
rp_{1}}{p_{1}-r},
\end{equation*}%
and

\begin{eqnarray*}
&&\left( \sum\limits_{j=1}^{n}\left\Vert A\left( e_{j}\right) \right\Vert ^{%
\frac{rp_{1}}{p_{1}-r}}\right) ^{\frac{1}{r}}=\left(
\sum\limits_{j=1}^{n}\left\Vert A\left( \left\Vert A\left( e_{j}\right)
\right\Vert ^{\frac{r}{p_{1}-r}}e_{j}\right) \right\Vert ^{r}\right) ^{\frac{%
1}{r}} \\
&\leq &C_{r}\left( Y\right) \left( \int_{0}^{1}\left\Vert
\sum_{j=1}^{n}r_{j}\left( t\right) \left\Vert A\left( e_{j}\right)
\right\Vert ^{\frac{r}{p_{1}-r}}A\left( e_{j}\right) \right\Vert
^{2}dt\right) ^{\frac{1}{2}} \\
&\leq &C_{r}\left( Y\right) \sup_{t\in \left[ 0,1\right] }\left\Vert
\sum_{j=1}^{n}r_{j}\left( t\right) \left\Vert A\left( e_{j}\right)
\right\Vert ^{\frac{r}{p_{1}-r}}A\left( e_{j}\right) \right\Vert \\
&\leq &C_{r}\left( Y\right) \sup_{\varphi \in B_{Y^{\ast
}}}\sum_{j=1}^{n}\left\vert \varphi \left( \left\Vert A\left( e_{j}\right)
\right\Vert ^{\frac{r}{p_{1}-r}}A\left( e_{j}\right) \right) \right\vert \\
&\leq &C_{r}\left( Y\right) \sup_{\varphi \in B_{Y^{\ast }}}\left(
\sum_{j=1}^{n}\left\Vert A\left( e_{j}\right) \right\Vert ^{\frac{rp_{1}}{%
p_{1}-r}}\right) ^{\frac{1}{p_{1}}}\left( \sum_{j=1}^{n}\left\vert \varphi
\left( A\left( e_{j}\right) \right) \right\vert ^{p_{1}^{\ast }}\right) ^{%
\frac{1}{p_{1}^{\ast }}} \\
&\leq &C_{r}\left( Y\right) \left( \sum_{j=1}^{n}\left\Vert A\left(
e_{j}\right) \right\Vert ^{\frac{rp_{1}}{p_{1}-r}}\right) ^{\frac{1}{p_{1}}%
}\left\Vert A\right\Vert .
\end{eqnarray*}%
So, if (b) is true, then (a) holds.

Assume (a). By the Maurey-Pisier factorization result (see \cite{pisier} and
\cite[pg. 286,287]{Di}) the infinite-dimensional Banach space $Y$ finitely
factors the formal inclusion $\ell _{r}\hookrightarrow \ell _{\infty }$,
i.e., there are constants $C_{1},C_{2}>0$ such that for all $n$ there are
vectors $z_{1},...,z_{n}\in Y$ satisfying%
\begin{equation*}
C_{1}\left\Vert \left( a_{j}\right) _{j=1}^{n}\right\Vert _{\infty }\leq
\left\Vert \sum\limits_{j=1}^{n}a_{j}z_{j}\right\Vert \leq C_{2}\left(
\sum\limits_{j=1}^{n}\left\vert a_{j}\right\vert ^{r}\right) ^{1/r}
\end{equation*}%
for all sequences of scalars $\left( a_{j}\right) _{j=1}^{n}.$ Consider the
continuous linear operator $A_{n}:\ell _{p_{1}}\rightarrow Y$ given by%
\begin{equation*}
A_{n}(x)=\sum\limits_{j=1}^{n}x_{j}z_{j}.
\end{equation*}%
Since%
\begin{equation*}
\frac{1}{p_{1}}+\frac{1}{\lambda _{1,r}^{p_{1}}}=\frac{1}{r},
\end{equation*}%
we have, using the H\"{o}lder inequality,%
\begin{equation*}
\left\Vert A_{n}\right\Vert =\sup_{\left\Vert x\right\Vert \leq 1}\left\Vert
\sum\limits_{j=1}^{n}x_{j}z_{j}\right\Vert \leq C_{2}n^{\frac{1}{\lambda
_{1,r}^{p_{1}}}}.
\end{equation*}%
On the other hand, there is a constant $C_{p_{1}}^{Y}=C$, such that%
\begin{equation*}
C\left\Vert A_{n}\right\Vert \geq \left( \sum_{j=1}^{n}\left\Vert
A_{n}(e_{j})\right\Vert ^{q_{1}}\right) ^{\frac{1}{q_{1}}}\geq C_{1}n^{\frac{%
1}{q_{1}}}.
\end{equation*}%
Since $n$ is arbitrary, $q_{1}\geq \lambda _{1,r}^{p_{1}}$ (i.e. (b) holds),
and this concludes the proof of the case $m=1.$

\bigskip

The proof of the general case is performed by induction on $m$. We know that
the result is valid for $m=1$ and we shall prove that it is valid for a
certain $m$ whenever it is valid for $m-1.$

(a)$\Rightarrow $(b). Let us suppose that
\begin{equation*}
\frac{1}{p_{1}}+...+\frac{1}{p_{m}}<\frac{1}{r}.
\end{equation*}%
\textit{A fortiori,}%
\begin{equation*}
\frac{1}{p_{2}}+...+\frac{1}{p_{m}}<\frac{1}{r}
\end{equation*}%
and, by our induction hypothesis, if there is a constant $%
C_{p_{2},...,p_{m}}^{Y}\geq 1$ such that%
\begin{equation*}
\left( \sum_{j_{2}=1}^{\infty }\left( \sum_{j_{3}=1}^{\infty }\cdots \left(
\sum_{j_{m}=1}^{\infty }\left\Vert A(e_{j_{2}},...,e_{j_{m}})\right\Vert
^{q_{m}}\right) ^{\frac{q_{m-1}}{q_{m}}}\cdots \right) ^{\frac{q_{2}}{q_{3}}%
}\right) ^{\frac{1}{q_{2}}}\leq C_{p_{2},...,p_{m}}^{Y}\left\Vert
A\right\Vert
\end{equation*}%
for all continuous $\left( m-1\right) $-linear operators $A:\ell
_{p_{2}}\times \cdots \times \ell _{p_{m}}\rightarrow Y$, then by Lemma \ref%
{propo} we conclude that (a) implies
\begin{eqnarray*}
q_{2} &\geq &\lambda _{m-1,r}^{p_{2},...,p_{m}}, \\
&&\vdots \\
q_{m-1} &\geq &\lambda _{2,r}^{p_{m-1},p_{m}}, \\
q_{m} &\geq &\lambda _{1,r}^{p_{m}}\text{.}
\end{eqnarray*}%
So, we must only show that
\begin{equation*}
q_{1}\geq \lambda _{m,r}^{p_{1},...,p_{m}}.
\end{equation*}%
As for the $m=1$ case, there are constants $C_{1},C_{2}>0$ such that for all
$n$ there are vectors $z_{1},...,z_{n}\in Y$ satisfying%
\begin{equation}
C_{1}\left\Vert \left( a_{j}\right) _{j=1}^{n}\right\Vert _{\infty }\leq
\left\Vert \sum\limits_{j=1}^{n}a_{j}z_{j}\right\Vert \leq C_{2}\left(
\sum\limits_{j=1}^{n}\left\vert a_{j}\right\vert ^{r}\right) ^{1/r}
\label{pisi}
\end{equation}%
for all sequences of scalars $\left( a_{j}\right) _{j=1}^{n}.$ Consider the
continuous multilinear operator $A_{n}:\ell _{p_{1}}\times \cdots \times
\ell _{p_{m}}\rightarrow Y$ given by%
\begin{equation*}
A_{n}(x^{\left( 1\right) },...,x^{\left( m\right)
})=\sum\limits_{j=1}^{n}x_{j}^{\left( 1\right) }x_{j}^{\left( 2\right)
}...x_{j}^{\left( m\right) }z_{j}.
\end{equation*}%
Since
\begin{equation*}
\frac{1}{\lambda _{m,r}^{p_{1},...,p_{m}}}+\sum_{k=1}^{m}\frac{1}{p_{k}}=%
\frac{1}{r}\text{,}
\end{equation*}%
by the H\"{o}lder inequality we obtain%
\begin{eqnarray*}
\left\Vert A_{n}\right\Vert &=&\sup_{\left\Vert x^{\left( 1\right)
}\right\Vert ,\cdots .\left\Vert x^{\left( m\right) }\right\Vert \leq
1}\left\Vert \sum\limits_{j=1}^{n}x_{j}^{\left( 1\right) }...x_{j}^{\left(
m\right) }z_{j}\right\Vert \leq \sup_{\left\Vert x^{\left( 1\right)
}\right\Vert ,\cdots ,\left\Vert x^{\left( m\right) }\right\Vert \leq
1}C_{2}\left( \sum\limits_{j=1}^{n}\left\vert x_{j}^{\left( 1\right)
}...x_{j}^{\left( m\right) }\right\vert ^{r}\right) ^{1/r} \\
&& \\
&\leq &\sup_{\left\Vert x^{\left( 1\right) }\right\Vert ,\cdots ,\left\Vert
x^{\left( m\right) }\right\Vert \leq 1}C_{2}\left( \prod_{k=1}^{m}\left(
\sum\limits_{j=1}^{n}\left\vert x_{j}^{\left( k\right) }\right\vert
^{p_{k}}\right) ^{1/p_{k}}\right) \left( \sum\limits_{j=1}^{n}\left\vert
1\right\vert ^{\lambda _{m,r}^{p_{1},...,p_{m}}}\right) ^{\frac{1}{\lambda
_{m,r}^{p_{1},...,p_{m}}}} \\
&\leq &C_{2}n^{\frac{1}{\lambda _{m,r}^{p_{1},...,p_{m}}}}.
\end{eqnarray*}%
On the other hand, by (\ref{pisi})
\begin{eqnarray*}
&&\left( \sum_{j_{1}=1}^{n}\left( \sum_{j_{2}=1}^{n}\cdots \left(
\sum_{j_{m}=1}^{n}\left\Vert A_{n}(e_{j_{1}},...,e_{j_{m}})\right\Vert
^{q_{m}}\right) ^{\frac{q_{m-1}}{q_{m}}}\cdots \right) ^{\frac{q_{1}}{q_{2}}%
}\right) ^{\frac{1}{q_{1}}} \\
&=&\left( \sum_{j=1}^{n}\left\Vert A_{n}(e_{j},...,e_{j})\right\Vert
^{q_{1}}\right) ^{\frac{1}{q_{1}}}=\left( \sum_{j=1}^{n}\left\Vert
z_{j}\right\Vert ^{q_{1}}\right) ^{\frac{1}{q_{1}}}\geq C_{1}n^{\frac{1}{%
q_{1}}},
\end{eqnarray*}%
and, since $n$ is arbitrary,
\begin{equation*}
q_{1}\geq \lambda _{m,r}^{p_{1},...,p_{m}}.
\end{equation*}

(b)$\Rightarrow $(a). Let $A:\ell _{p_{1}}\times \cdots \times \ell
_{p_{m}}\rightarrow Y$ be a continuous $m$-linear operator and define, for
all positive integers $n$,%
\begin{equation*}
A_{n,e}:\ell _{p_{1}}\times \cdots \times \ell _{p_{m-1}}\rightarrow \ell
_{\lambda _{1,r}^{p_{m}}}\left( Y\right)
\end{equation*}%
by%
\begin{equation*}
A_{n,e}(x^{\left( 1\right) },...,x^{\left( m-1\right) })=\left( A\left(
x^{\left( 1\right) },...,x^{\left( m-1\right) },e_{j}\right) \right)
_{j=1}^{n}.
\end{equation*}%
We assert that
\begin{equation*}
\left\Vert A_{n,e}\right\Vert \leq C_{r}\left( Y\right) \left\Vert
A\right\Vert .
\end{equation*}%
To see this, since $Y$ has cotype $r$ and using the H\"{o}lder inequality,
\begin{eqnarray*}
&&\left( \sum\limits_{j=1}^{n}\left\Vert A\left( x^{\left( 1\right)
},...,x^{\left( m-1\right) },e_{j}\right) \right\Vert ^{\frac{rp_{m}}{p_{m}-r%
}}\right) ^{\frac{1}{r}} \\
&=&\left( \sum\limits_{j=1}^{n}\left\Vert A\left( x^{\left( 1\right)
},...,x^{\left( m-1\right) },\left\Vert A\left( x^{\left( 1\right)
},...,x^{\left( m-1\right) },e_{j}\right) \right\Vert ^{\frac{r}{p_{m}-r}%
}e_{j}\right) \right\Vert ^{r}\right) ^{\frac{1}{r}} \\
&\leq &C_{r}\left( Y\right) \left( \int_{0}^{1}\left\Vert
\sum_{j=1}^{n}r_{j}\left( t\right) \left\Vert A\left( x^{\left( 1\right)
},...,x^{\left( m-1\right) },e_{j}\right) \right\Vert ^{\frac{r}{p_{m}-r}%
}A\left( x^{\left( 1\right) },...,x^{\left( m-1\right) },e_{j}\right)
\right\Vert ^{2}dt\right) ^{\frac{1}{2}} \\
&\leq &C_{r}\left( Y\right) \sup_{t\in \left[ 0,1\right] }\left\Vert
\sum_{j=1}^{n}r_{j}\left( t\right) \left\Vert A\left( x^{\left( 1\right)
},...,x^{\left( m-1\right) },e_{j}\right) \right\Vert ^{\frac{r}{p_{m}-r}%
}A\left( x^{\left( 1\right) },...,x^{\left( m-1\right) },e_{j}\right)
\right\Vert \\
&\leq &C_{r}\left( Y\right) \sup_{\varphi \in B_{Y^{\ast
}}}\sum_{j=1}^{n}\left\vert \varphi \left( \left\Vert A\left( x^{\left(
1\right) },...,x^{\left( m-1\right) },e_{j}\right) \right\Vert ^{\frac{r}{%
p_{m}-r}}A\left( x^{\left( 1\right) },...,x^{\left( m-1\right)
},e_{j}\right) \right) \right\vert \\
&\leq &C_{r}\left( Y\right) \sup_{\varphi \in B_{Y^{\ast }}}\left(
\sum_{j=1}^{n}\left\Vert A\left( x^{\left( 1\right) },...,x^{\left(
m-1\right) },e_{j}\right) \right\Vert ^{\frac{rp_{m}}{p_{m}-r}}\right) ^{%
\frac{1}{p_{m}}}\left( \sum_{j=1}^{n}\left\vert \varphi \left( A\left(
x^{\left( 1\right) },...,x^{\left( m-1\right) },e_{j}\right) \right)
\right\vert ^{p_{m}^{\ast }}\right) ^{\frac{1}{p_{m}^{\ast }}} \\
&\leq &C_{r}\left( Y\right) \left( \sum_{j=1}^{n}\left\Vert A\left(
x^{\left( 1\right) },...,x^{\left( m-1\right) },e_{j}\right) \right\Vert ^{%
\frac{rp_{m}}{p_{m}-r}}\right) ^{\frac{1}{p_{m}}}\left\Vert A\left(
x^{\left( 1\right) },...,x^{\left( m-1\right) },\cdot \right) \right\Vert .
\end{eqnarray*}%
Therefore,%
\begin{equation*}
\left( \sum\limits_{j=1}^{n}\left\Vert A\left( x^{\left( 1\right)
},...,x^{\left( m-1\right) },e_{j}\right) \right\Vert ^{\frac{rp_{m}}{p_{m}-r%
}}\right) ^{\frac{p_{m}-r}{rp_{m}}}\leq C_{r}\left( Y\right) \left\Vert
A\right\Vert \left\Vert x^{\left( 1\right) }\right\Vert \cdots \left\Vert
x^{\left( m-1\right) }\right\Vert
\end{equation*}%
and thus
\begin{eqnarray*}
&&\left\Vert A_{n,e}\right\Vert =\sup_{\left\Vert x^{\left( 1\right)
}\right\Vert ,\cdots ,\left\Vert x^{\left( m-1\right) }\right\Vert \leq
1}\left\Vert A_{n,e}\left( x^{\left( 1\right) },...,x^{\left( m-1\right)
}\right) \right\Vert \\
&=&\sup_{\left\Vert x^{\left( 1\right) }\right\Vert ,\cdots ,\left\Vert
x^{\left( m-1\right) }\right\Vert \leq 1}\left(
\sum\limits_{j=1}^{n}\left\Vert A\left( x^{\left( 1\right) },...,x^{\left(
m-1\right) },e_{j}\right) \right\Vert ^{\frac{rp_{m}}{p_{m}-r}}\right) ^{%
\frac{p_{m}-r}{rp_{m}}} \\
&\leq &C_{r}\left( Y\right) \left\Vert A\right\Vert ,
\end{eqnarray*}%
as required.

On the other hand, since $X=\ell _{\lambda _{1,r}^{p_{m}}}\left( Y\right) $
has cotype $\lambda _{1,r}^{p_{m}}:=R$ (because $\lambda
_{1,r}^{p_{m}}>r=\cot Y$) and%
\begin{equation*}
\frac{1}{p_{1}}+...+\frac{1}{p_{m-1}}<\frac{1}{\cot Y}-\frac{1}{p_{m}}=\frac{%
1}{\cot X},
\end{equation*}%
we can use the induction hypothesis (with the $\left( m-1\right) $-linear
operator $A_{n,e}$), and conclude that if
\begin{equation*}
q_{1}\geq \lambda _{m-1,R}^{p_{1},...,p_{m-1}},q_{2}\geq \lambda
_{m-2,R}^{p_{2},...,p_{m-1}},...,q_{m-1}\geq \lambda _{1,R}^{p_{m-1}},
\end{equation*}%
then
\begin{eqnarray*}
&&\left( \sum_{j_{1}=1}^{n}\left( \sum_{j_{2}=1}^{n}\cdots \left(
\sum_{j_{m-1}=1}^{n}\left( \sum_{j_{m}=1}^{n}\left\Vert
A(e_{j_{1}},...,e_{j_{m}})\right\Vert ^{R}\right) ^{\frac{q_{m-1}}{R}%
}\right) ^{\frac{q_{m-2}}{q_{m-1}}}\cdots \right) ^{\frac{q_{1}}{q_{2}}%
}\right) ^{\frac{1}{q_{1}}} \\
&=&\left( \sum_{j_{1}=1}^{n}\left( \sum_{j_{2}=1}^{n}\cdots \left(
\sum_{j_{m-1}=1}^{n}\left\Vert A_{n,e}(e_{j_{1}},...,e_{j_{m-1}})\right\Vert
_{X}^{q_{m-1}}\right) ^{\frac{q_{m-2}}{q_{m-1}}}\cdots \right) ^{\frac{q_{1}%
}{q_{2}}}\right) ^{\frac{1}{q_{1}}} \\
&\leq &C_{p_{1},...,p_{m-1}}^{X}\left\Vert A_{n,e}\right\Vert \\
&\leq &C_{p_{1},...,p_{m-1}}^{X}C_{r}\left( Y\right) \left\Vert A\right\Vert
.
\end{eqnarray*}%
Now, the proof is almost done, since%
\begin{eqnarray*}
\lambda _{m-k,R}^{p_{k},...,p_{m-1}} &=&\frac{R}{1-R\left( \frac{1}{p_{k}}+%
\frac{1}{p_{k+1}}+...+\frac{1}{p_{m-1}}\right) } \\
&=&\frac{\lambda _{1,r}^{p_{m}}}{1-\lambda _{1,r}^{p_{m}}\left( \frac{1}{%
p_{k}}+\frac{1}{p_{k+1}}+...+\frac{1}{p_{m-1}}\right) } \\
&=&\frac{\frac{rp_{m}}{p_{m}-r}}{1-\frac{rp_{m}}{p_{m}-r}\left( \frac{1}{%
p_{k}}+\frac{1}{p_{k+1}}+...+\frac{1}{p_{m-1}}\right) } \\
&=&\lambda _{m-k+1,r}^{p_{k},...,p_{m}}
\end{eqnarray*}%
for each $k\in \left\{ 1,...,m-1\right\} $.

To conclude the proof we just need to remark that
\begin{eqnarray*}
&&\left( \sum_{j_{1}=1}^{\infty }\left( \sum_{j_{2}=1}^{\infty }\cdots
\left( \sum_{j_{m-1}=1}^{\infty }\left( \sum_{j_{m}=1}^{\infty }\left\Vert
A(e_{j_{1}},...,e_{j_{m}})\right\Vert ^{q_{m}}\right) ^{\frac{q_{m-1}}{q_{m}}%
}\right) ^{\frac{q_{m-2}}{q_{m-1}}}\cdots \right) ^{\frac{q_{1}}{q_{2}}%
}\right) ^{\frac{1}{q_{1}}} \\
&\leq &\left( \sum_{j_{1}=1}^{\infty }\left( \sum_{j_{2}=1}^{\infty }\cdots
\left( \sum_{j_{m-1}=1}^{\infty }\left( \sum_{j_{m}=1}^{\infty }\left\Vert
A(e_{j_{1}},...,e_{j_{m}})\right\Vert ^{R}\right) ^{\frac{q_{m-1}}{R}%
}\right) ^{\frac{q_{m-2}}{q_{m-1}}}\cdots \right) ^{\frac{q_{1}}{q_{2}}%
}\right) ^{\frac{1}{q_{1}}}
\end{eqnarray*}%
provided $q_{m}\geq R=\lambda _{1,r}^{p_{m}}$.
\end{proof}

\bigskip

Since%
\begin{equation*}
\lambda _{m,r}^{p_{1},...,p_{m}}\geq \lambda _{m-1,r}^{p_{2},...,p_{m}}\geq
\cdots \geq \lambda _{2,r}^{p_{m-1},p_{m}}\geq \lambda _{1,r}^{p_{m}}
\end{equation*}%
and
\begin{equation*}
\frac{1}{\lambda _{m,r}^{p_{1},...,p_{m}}}=\frac{1-r\left( \frac{1}{p_{1}}+%
\frac{1}{p_{2}}+...+\frac{1}{p_{m}}\right) }{r}=\frac{1}{r}-\left( \frac{1}{%
p_{1}}+\frac{1}{p_{2}}+...+\frac{1}{p_{m}}\right) ,
\end{equation*}%
the previous theorem generalizes Proposition 4.3 from \cite{dimant} and
Theorem 1.5 of \cite{abps2}, now with optimal exponents in a stronger sense.

\smallskip

\begin{corollary}
Let $Y$ be an infinite-dimensional Banach space with cotype $\cot Y$ and $%
p_{1},...,p_{m}>\cot Y$, such that
\begin{equation*}
\frac{1}{p_{1}}+...+\frac{1}{p_{m}}<\frac{1}{\cot Y}.
\end{equation*}%
Then there is a constant $B_{p_{1},...,p_{m}}^{Y}\geq 1$ such that%
\begin{equation*}
\left( \sum_{j_{1},...,j_{m}=1}^{\infty }\left\Vert
A(e_{j_{1}},...,e_{j_{m}})\right\Vert ^{\lambda _{m,\cot
Y}^{p_{1},...,p_{m}}}\right) ^{\frac{1}{\lambda _{m,\cot Y}^{p_{1},...,p_{m}}%
}}\leq B_{p_{1},...,p_{m}}^{Y}\left\Vert A\right\Vert
\end{equation*}%
for all continuous $m$-linear operators $A:\ell _{p_{1}}\times \cdots \times
\ell _{p_{m}}\rightarrow Y.$
\end{corollary}

\bigskip In the case that we do not know if $Y$ attains the infimum of its
cotypes, using the previous arguments, it is possible to prove the following:

\smallskip

\begin{theorem}
\label{662} Let $q_{1},...,q_{m}>0$ and $Y$ be an infinite-dimensional
Banach space with finite cotype. If
\begin{equation*}
\frac{1}{p_{1}}+...+\frac{1}{p_{m}}<\frac{1}{\cot Y},
\end{equation*}%
then the following assertions are equivalent:

(a) There is a constant $C_{p_{1},...,p_{m}}^{Y,\varepsilon }\geq 1$ such
that%
\begin{equation*}
\left( \sum_{j_{1}=1}^{\infty }\left( \sum_{j_{2}=1}^{\infty }\cdots \left(
\sum_{j_{m}=1}^{\infty }\left\Vert A(e_{j_{1}},...,e_{j_{m}})\right\Vert
^{q_{m}+\varepsilon }\right) ^{\frac{q_{m-1}+\varepsilon }{q_{m}+\varepsilon
}}\cdots \right) ^{\frac{q_{1}+\varepsilon }{q_{2}+\varepsilon }}\right) ^{%
\frac{1}{q_{1}+\varepsilon }}\leq C_{p_{1},...,p_{m}}^{Y,\varepsilon
}\left\Vert A\right\Vert
\end{equation*}%
for all continuous $m$-linear operators $A:\ell _{p_{1}}\times \cdots \times
\ell _{p_{m}}\rightarrow Y$, and all $\varepsilon >0.$

(b) The exponents $q_{1},...,q_{m}$ satisfy
\begin{equation*}
q_{1}\geq \lambda _{m,\cot Y}^{p_{1},...,p_{m}},q_{2}\geq \lambda _{m-1,\cot
Y}^{p_{2},...,p_{m}},...,q_{m-1}\geq \lambda _{2,\cot
Y}^{p_{m-1},p_{m}},q_{m}\geq \lambda _{1,\cot Y}^{p_{m}}.
\end{equation*}
\end{theorem}

\smallskip

\begin{remark}
Analogous results obtained by permuting the indices in Theorems \ref{661}
and \ref{662} hold with suitable modifications on the conditions for the
exponents.
\end{remark}

\section{Optimal exponents: scalar-valued case\label{s55}}

In this section we prove a (sharp) multilinear generalization of Theorems %
\ref{hl3} and \ref{t1}. Let $p_{1},...,p_{m}>1$, such that $\frac{1}{p_{1}}+%
\frac{1}{p_{2}}+...+\frac{1}{p_{m}}<1$. For all positive integers $m $ and $%
k=1,...,m,$ let us define
\begin{equation*}
\delta _{m-k+1}^{p_{k},...,p_{m}}:=\frac{1}{1-\left( \frac{1}{p_{k}}+...+%
\frac{1}{p_{m}}\right) }.
\end{equation*}

As we will see, the proof of the following lemma is similar to the proof of (a)$\Rightarrow $(b) of Theorem \ref{661}. In fact, it is somewhat simpler here, since no appeal to the Maurey-Pisier factorization result is needed.

\begin{lemma}
\label{t0multi} Let $m$ be a positive integer, $q_{1},...,q_{m}>0$, and $p_{1},...,p_{m}>1$, with
\begin{equation*}
\frac{1}{p_{1}}+\frac{1}{p_{2}}+...+\frac{1}{p_{m}}<1.
\end{equation*}If there is a constant $C_{p_{1},...,p_{m}}\geq 1$ such that\begin{equation*}
\left( \sum_{j_{1}=1}^{\infty }\left( \sum_{j_{2}=1}^{\infty }\cdots \left(
\sum_{j_{m}=1}^{\infty }\left\vert A(e_{j_{1}},...,e_{j_{m}})\right\vert
^{q_{m}}\right) ^{\frac{q_{m-1}}{q_{m}}}\cdots \right) ^{\frac{q_{1}}{q_{2}}}\right) ^{\frac{1}{q_{1}}}\leq C_{p_{1},...,p_{m}}\left\Vert A\right\Vert
\end{equation*}for all continuous $m$-linear operators $A:\ell _{p_{1}}\times \cdots \times
\ell _{p_{m}}\rightarrow \mathbb{K}$, then the exponents $q_{1},...,q_{m}$
satisfy
\begin{equation*}
q_{1}\geq \delta _{m}^{p_{1},...,p_{m}},q_{2}\geq \delta
_{m-1}^{p_{2},...,p_{m}},...,q_{m-1}\geq \delta
_{2}^{p_{m-1},p_{m}},q_{m}\geq \delta _{1}^{p_{m}}.
\end{equation*}
\end{lemma}

\begin{proof}
Let $p>1$ and $q>0.$ It is well known that if there is a constant $C_{p}\geq
1$ such that
\begin{equation*}
\left( \sum\nolimits_{j_{1}=1}^{\infty }\left\vert A(e_{j_{1}})\right\vert
^{q}\right) ^{\frac{1}{q}}\leq C_{p}\left\Vert A\right\Vert
\end{equation*}%
for all continuous linear operators $A:\ell _{p}\rightarrow \mathbb{K}$,
then $q\geq \delta _{1}^{p};$ thus the case $m=1$, is done.

Let us suppose the case $m-1$ and prove the case $m$ by induction. By
assumption if there is a constant $C_{p_{2},...,p_{m}}\geq 1$ such that%
\begin{equation*}
\left( \sum_{j_{2}=1}^{\infty }\left( \sum_{j_{3}=1}^{\infty }\cdots \left(
\sum_{j_{m}=1}^{\infty }\left\vert A(e_{j_{2}},...,e_{j_{m}})\right\vert
^{q_{m}}\right) ^{\frac{q_{m-1}}{q_{m}}}\cdots \right) ^{\frac{q_{2}}{q_{3}}%
}\right) ^{\frac{1}{q_{2}}}\leq C_{p_{2},...,p_{m}}\left\Vert A\right\Vert
\end{equation*}%
for all continuous $\left( m-1\right) $-linear forms $A:\ell _{p_{2}}\times
\cdots \times \ell _{p_{m}}\rightarrow \mathbb{K}$, then by Lemma \ref{propo}%
, (a) implies
\begin{eqnarray*}
q_{2} &\geq &\delta _{m-1}^{p_{2},...,p_{m}} \\
&&\vdots \\
q_{m-1} &\geq &\delta _{2}^{p_{m-1},p_{m}} \\
q_{m} &\geq &\delta _{1}^{p_{m}}\text{.}
\end{eqnarray*}
It remains to estimate $q_{1}.$ For each $n$ consider the continuous
multilinear form $A_{n}:\ell _{p_{1}}\times \cdots \times \ell
_{p_{m}}\rightarrow \mathbb{K}$ given by%
\begin{equation*}
A_{n}(x^{\left( 1\right) },...,x^{\left( m\right)
})=\sum\limits_{j=1}^{n}x_{j}^{\left( 1\right) }x_{j}^{\left( 2\right)
}...x_{j}^{\left( m\right) }.
\end{equation*}%
Since
\begin{equation*}
\frac{1}{\delta _{m}^{p_{1},\cdots ,p_{m}}}+\sum_{k=1}^{m}\frac{1}{p_{k}}=1%
\text{,}
\end{equation*}%
we use the H\"{o}lder inequality and obtain%
\begin{eqnarray*}
\left\Vert A_{n}\right\Vert &=&\sup_{\left\Vert x^{\left( 1\right)
}\right\Vert ,\cdots ,\left\Vert x^{\left( m\right) }\right\Vert \leq
1}\left\vert \sum\limits_{j=1}^{n}x_{j}^{\left( 1\right) }x_{j}^{\left(
2\right) }...x_{j}^{\left( m\right) }\right\vert \\
&\leq &\sup_{\left\Vert x^{\left( 1\right) }\right\Vert ,\cdots ,\left\Vert
x^{\left( m\right) }\right\Vert \leq 1}\left( \prod_{k=1}^{m}\left(
\sum\limits_{j=1}^{n}\left\vert x_{j}^{\left( k\right) }\right\vert
^{p_{k}}\right) ^{1/p_{k}}\left( \sum\limits_{j=1}^{n}\left\vert
1\right\vert ^{\delta _{m}^{p_{1}\cdots p_{m}}}\right) ^{\frac{1}{\delta
_{m}^{p_{1},\cdots ,p_{m}}}}\right) \\
&\leq &n^{\frac{1}{\delta _{m}^{p_{1},\cdots ,p_{m}}}}.
\end{eqnarray*}%
On the other hand%
\begin{equation*}
\left( \sum_{j_{1}=1}^{n}\left( \sum_{j_{2}=1}^{n}\cdots \left(
\sum_{j_{m}=1}^{n}\left\vert A_{n}(e_{j_{1}},...,e_{j_{m}})\right\vert
^{q_{m}}\right) ^{\frac{q_{m-1}}{q_{m}}}\cdots \right) ^{\frac{q_{1}}{q_{2}}%
}\right) ^{\frac{1}{q_{1}}}=n^{\frac{1}{q_{1}}},
\end{equation*}%
and, since $n$ is arbitrary,
\begin{equation*}
q_{1}\geq \delta _{m}^{p_{1},...,p_{m}}.
\end{equation*}
\end{proof}

\bigskip

The next theorem is the main result of this section. It is a consequence of
our Theorem \ref{661}, and generalizes Theorems \ref{hl3} and \ref{t1}. The
reader should note that the hypothesis $1<p_{m}\leq 2< p_{1},...,p_{m-1}$ is
quite natural, along the lines of a generalization of these Theorems. In
fact, if we had $p_{i},p_{j}\leq 2$ for some $i,j$, then we would have
\begin{equation*}
\frac{1}{p_{1}}+\frac{1}{p_{2}}+...+\frac{1}{p_{m}}\geq 1,
\end{equation*}%
and this is not the environment of a generalization of Theorems \ref{hl3}
and \ref{t1}.

\begin{theorem}
\label{t1multi} Let $m\geq 2$, $q_{1},...,q_{m}>0$, and $1<p_{m}\leq
2<p_{1},...,p_{m-1}$, with
\begin{equation*}
\frac{1}{p_{1}}+\frac{1}{p_{2}}+...+\frac{1}{p_{m}}<1.
\end{equation*}%
The following assertions are equivalent:

(a) There is a constant $C_{p_{1},...,p_{m}}\geq 1$ such that%
\begin{equation*}
\left( \sum_{j_{1}=1}^{\infty }\left( \sum_{j_{2}=1}^{\infty }\cdots \left(
\sum_{j_{m}=1}^{\infty }\left\vert A(e_{j_{1}},...,e_{j_{m}})\right\vert
^{q_{m}}\right) ^{\frac{q_{m-1}}{q_{m}}}\cdots \right) ^{\frac{q_{1}}{q_{2}}%
}\right) ^{\frac{1}{q_{1}}}\leq C_{p_{1},...,p_{m}}\left\Vert A\right\Vert
\end{equation*}%
for all continuous $m$-linear operators $A:\ell _{p_{1}}\times \cdots \times
\ell _{p_{m}}\rightarrow \mathbb{K}$.

(b) The exponents $q_{1},...,q_{m}>0$ satisfy
\begin{equation*}
q_{1}\geq \delta _{m}^{p_{1},...,p_{m}},q_{2}\geq \delta
_{m-1}^{p_{2},...,p_{m}},...,q_{m-1}\geq \delta
_{2}^{p_{m-1},p_{m}},q_{m}\geq \delta _{1}^{p_{m}}.
\end{equation*}
\end{theorem}

\begin{proof}
{(a)$\Rightarrow $(b) is a particular case of Lemma \ref{t0multi}.}\newline
(b)$\Rightarrow $(a). Let $A:\ell _{p_{1}}\times \cdots \times \ell
_{p_{m}}\rightarrow \mathbb{K}$ be a continuous $m$-linear operator and
define, for all positive integers $n$,%
\begin{equation*}
A_{n,e}:\ell _{p_{1}}\times \cdots \times \ell _{p_{m-1}}\rightarrow \ell
_{\delta _{1}^{p_{m}}}
\end{equation*}%
by%
\begin{equation*}
A_{n,e}(x^{\left( 1\right) },...,x^{\left( m-1\right) })=\left( A\left(
x^{\left( 1\right) },...,x^{\left( m-1\right) },e_{j}\right) \right)
_{j=1}^{n}.
\end{equation*}%
Note that
\begin{equation*}
\left\Vert A_{n,e}\right\Vert \leq \left\Vert A\right\Vert .
\end{equation*}%
In fact, we obviously have
\begin{equation*}
\left( \sum\limits_{j=1}^{n}\left\vert A\left( x^{\left( 1\right)
},...,x^{\left( m-1\right) },e_{j}\right) \right\vert ^{\frac{p_{m}}{p_{m}-1}%
}\right) ^{\frac{p_{m}-1}{p_{m}}}\leq \left\Vert A\right\Vert \left\Vert
x^{\left( 1\right) }\right\Vert \cdots \left\Vert x^{\left( m-1\right)
}\right\Vert .
\end{equation*}%
Therefore
\begin{eqnarray*}
&&\left\Vert A_{n,e}\right\Vert =\sup_{\left\Vert x^{\left( 1\right)
}\right\Vert \cdots \left\Vert x^{\left( m-1\right) }\right\Vert \leq
1}\left\Vert A_{n,e}\left( x^{\left( 1\right) },...,x^{\left( m-1\right)
}\right) \right\Vert \\
&=&\sup_{\left\Vert x^{\left( 1\right) }\right\Vert \cdots \left\Vert
x^{\left( m-1\right) }\right\Vert \leq 1}\left(
\sum\limits_{j=1}^{n}\left\vert A\left( x^{\left( 1\right) },...,x^{\left(
m-1\right) },e_{j}\right) \right\vert ^{\frac{p_{m}}{p_{m}-1}}\right) ^{%
\frac{p_{m}-1}{p_{m}}} \\
&\leq &\left\Vert A\right\Vert .
\end{eqnarray*}%
On the other hand, since $\ell _{\delta _{1}^{p_{m}}}$ has cotype $\delta
_{1}^{p_{m}}:=r$ (because $p_{m}\leq 2$) and
\begin{equation*}
\frac{1}{p_{1}}+...+\frac{1}{p_{m-1}}<1-\frac{1}{p_{m}}=\frac{1}{r},
\end{equation*}%
we can invoke Theorem \ref{661} for $\left( m-1\right) $-linear operators.
Thus, if
\begin{equation*}
q_{1}\geq \lambda _{m-1,r}^{p_{1},...,p_{m-1}},q_{2}\geq \lambda
_{m-2,r}^{p_{2},...,p_{m-1}},...,q_{m-1}\geq \lambda _{1,r}^{p_{m-1}},
\end{equation*}%
we have
\begin{eqnarray*}
&&\left( \sum_{j_{1}=1}^{n}\left( \sum_{j_{2}=1}^{n}\cdots \left(
\sum_{j_{m-1}=1}^{n}\left( \sum_{j_{m}=1}^{n}\left\vert
A(e_{j_{1}},...,e_{j_{m}})\right\vert ^{r}\right) ^{\frac{q_{m-1}}{r}%
}\right) ^{\frac{q_{m-2}}{q_{m-1}}}\cdots \right) ^{\frac{q_{1}}{q_{2}}%
}\right) ^{\frac{1}{q_{1}}} \\
&=&\left( \sum_{j_{1}=1}^{n}\left( \sum_{j_{2}=1}^{n}\cdots \left(
\sum_{j_{m-1}=1}^{n}\left\Vert A_{n,e}(e_{j_{1}},...,e_{j_{m-1}})\right\Vert
_{\ell _{\delta _{1}^{p_{m}}}}^{q_{m-1}}\right) ^{\frac{q_{m-2}}{q_{m-1}}%
}\cdots \right) ^{\frac{q_{1}}{q_{2}}}\right) ^{\frac{1}{q_{1}}} \\
&\leq &C_{p_{1},...,p_{m-1}}^{\ell _{\delta _{1}^{p_{m}}}}\left\Vert
A_{n,e}\right\Vert \\
&\leq &C_{p_{1},...,p_{m-1}}^{\ell _{\delta _{1}^{p_{m}}}}\left\Vert
A\right\Vert .
\end{eqnarray*}%
Since%
\begin{eqnarray*}
\lambda _{m-k,r}^{p_{k},...,p_{m-1}} &=&\frac{\delta _{1}^{p_{m}}}{1-\delta
_{1}^{p_{m}}\left( \frac{1}{p_{k}}+\frac{1}{p_{k+1}}+...+\frac{1}{p_{m-1}}%
\right) } \\
&=&\delta _{m-k+1}^{p_{k},...,p_{m}}
\end{eqnarray*}%
for each $k\in \left\{ 1,...,m-1\right\} $, and

\begin{eqnarray*}
&&\left( \sum_{j_{1}=1}^{\infty }\left( \sum_{j_{2}=1}^{\infty }\cdots
\left( \sum_{j_{m-1}=1}^{\infty }\left( \sum_{j_{m}=1}^{\infty }\left\vert
A(e_{j_{1}},...,e_{j_{m}})\right\vert ^{q_{m}}\right) ^{\frac{q_{m-1}}{q_{m}}%
}\right) ^{\frac{q_{m-2}}{q_{m-1}}}\cdots \right) ^{\frac{q_{1}}{q_{2}}%
}\right) ^{\frac{1}{q_{1}}} \\
&\leq &\left( \sum_{j_{1}=1}^{\infty }\left( \sum_{j_{2}=1}^{\infty }\cdots
\left( \sum_{j_{m-1}=1}^{\infty }\left( \sum_{j_{m}=1}^{\infty }\left\vert
A(e_{j_{1}},...,e_{j_{m}})\right\vert ^{r}\right) ^{\frac{q_{m-1}}{r}%
}\right) ^{\frac{q_{m-2}}{q_{m-1}}}\cdots \right) ^{\frac{q_{1}}{q_{2}}%
}\right) ^{\frac{1}{q_{1}}}
\end{eqnarray*}%
provided $q_{m}\geq r=\delta _{1}^{p_{m}}$, the proof is done.
\end{proof}

\bigskip

Since%
\begin{equation*}
\delta _{m}^{p_{1},...,p_{m}}\geq \delta _{m-1}^{p_{2},...,p_{m}}\geq \cdots
\geq \delta _{2}^{p_{m-1},p_{m}}\geq \delta _{1}^{p_{m}},
\end{equation*}%
and
\begin{equation*}
\delta _{m}^{p_{1},...,p_{m}}=\frac{1}{1-\left( \frac{1}{p_{1}}+\frac{1}{%
p_{2}}+...+\frac{1}{p_{m}}\right) },
\end{equation*}%
then the previous theorem generalizes Theorems \ref{hl3} and \ref{t1}, with
optimal exponents for the multilinear form case.

\begin{corollary}
\label{555666} Let $m\geq 2$, $1<p_{m}\leq 2<p_{1},...,p_{m-1}$, with
\begin{equation*}
\frac{1}{p_{1}}+...+\frac{1}{p_{m}}<1.
\end{equation*}%
Then there is a constant $C_{p_{1},...,p_{m}}\geq 1$ such that%
\begin{equation}
\left( \sum_{j_{1},...,j_{m}=1}^{\infty }\left\vert
A(e_{j_{1}},...,e_{j_{m}})\right\vert ^{\frac{1}{1-\left( \frac{1}{p_{1}}%
+...+\frac{1}{p_{m}}\right) }}\right) ^{1-\left( \frac{1}{p_{1}}+...+\frac{1%
}{p_{m}}\right) }\leq C_{p_{1},...,p_{m}}\left\Vert A\right\Vert  \label{q11}
\end{equation}%
for all continuous $m$-linear operators $A:\ell _{p_{1}}\times \cdots \times
\ell _{p_{m}}\rightarrow \mathbb{K}$.
\end{corollary}

In the final section we show that the optimal constant $C_{p_{1},...,p_{m}}$
of Theorem \ref{t1multi} and Corollary \ref{555666} is precisely $1.$

\section{Optimal constants\label{s4}}

The Banach spaces in this section are considered over the complex scalar
field. Let us begin by recalling that the Rademacher matrices $R_{n}=\left(
r_{ij}^{\left( n\right) }\right) $, $i=1,...,2^{n}$, $j=1,...,n$,~are the $%
2^{n}\times n$ matrices defined recursively as follows:%
\begin{equation*}
R_{1}=\left(
\begin{array}{c}
1 \\
-1%
\end{array}%
\right) ,\text{ }R_{n+1}=\left(
\begin{tabular}{l|l}
$%
\begin{array}{c}
1 \\
\vdots \\
1%
\end{array}%
$ & $R_{n}$ \\ \hline
$%
\begin{array}{c}
-1 \\
\vdots \\
-1%
\end{array}%
$ & $R_{n}$%
\end{tabular}%
\right) ,
\end{equation*}%
for $n\in \mathbb{N}$. Note that $r_{ij}^{\left( n\right) }=r_{j}\left(
\frac{2i-1}{2^{n+1}}\right) $, where $r_{j}$ denotes the $j$-th Rademacher
function.

Let $1\leq p\leq 2$ and $0<s<\infty $. Recall that a Banach space $X$ has
\emph{type} $p$ (see \cite{kato}) if there is a constant $C>0$ such that, no
matter how we select finitely many vectors $x_{1},\dots ,x_{n}\in X$,%
\begin{equation}
\left( \int_{\lbrack 0,1]}\left\Vert \sum_{k=1}^{n}r_{k}(t)x_{k}\right\Vert
^{s}dt\right) ^{1/s}\leq C\left( \sum_{k=1}^{n}\Vert x_{k}\Vert ^{p}\right)
^{\frac{1}{p}},  \label{typo}
\end{equation}%
where $r_{k}$ denotes the $k$-th Rademacher function. It is well known that
if (\ref{typo}) is satisfied for a certain $s>0$, then it is satisfied for
all $s>0.$ For a fixed $s$, the smallest of all constants $C$ will be
denoted by $T_{p,s}\left( X\right) $.

In the following result of \cite{kato}, type and cotype properties are
described via the linear operators induced by the Rademacher matrices and
their transposes:

\begin{proposition}
\label{propojapa}(See \cite[Proposition 2.3]{kato}) Let $X$ be a Banach
space.

(i) Let $1<p\leq 2$. Then $X$ has type $p$ if and only if there exist some $%
s $, $1\leq s<\infty ~,$ and a constant $M$ such that%
\begin{equation*}
\left\Vert R_{n}:\ell _{p}^{n}\left( X\right) \rightarrow \ell
_{s}^{2^{n}}\left( X\right) \right\Vert \leq M2^{\frac{n}{s}},
\end{equation*}%
for all $n\in \mathbb{N}$. Moreover, $T_{p,s}\left( X\right) \leq M$.

(ii) Let $2\leq q<\infty $, and $t_{R_{n}}$ be the transposed matrix of $%
R_{n}$. Then $X$ has cotype $q$ if and only if there exist some $s$, $1\leq
s<\infty ,$ and a constant $M$ such that%
\begin{equation*}
\left\Vert t_{R_{n}}:\ell _{s}^{2^{n}}\left( X\right) \rightarrow \ell
_{q}^{n}\left( X\right) \right\Vert \leq M2^{\frac{n}{s^{\ast }}},
\end{equation*}%
for all $n\in \mathbb{N}$. Moreover, $C_{q,s}\left( X\right) \leq M$.
\end{proposition}

\begin{lemma}
\bigskip \label{lemmajapa}(See \cite[Lemma 2.3]{kato}) Let $H$ be a Hilbert
space. Then%
\begin{equation*}
2^{\frac{n}{2}}=\left\Vert R_{n}:\ell _{2}^{n}\left( H\right) \rightarrow
\ell _{2}^{2^{n}}\left( H\right) \right\Vert =\left\Vert t_{R_{n}}:\ell
_{2}^{2^{n}}\left( H\right) \rightarrow \ell _{2}^{n}\left( H\right)
\right\Vert ,
\end{equation*}%
for all $n\in \mathbb{N}$.
\end{lemma}

\bigskip Let us introduce the following notation: for $1\leq p_{1}< \infty $
and $1\leq p_{2}\leq \infty $, we denote by $\ell _{p_{1}}\left( \ell
_{p_{2}}\right) $ the Banach space of the sequences $\mathbf{x}=\left(
x_{i_{1},i_{2}}\right) _{i_{1},i_{2}=1}^{\infty }$ such that
\begin{equation*}
\left\Vert \mathbf{x}\right\Vert _{\ell _{p_{1}}\left( \ell _{p_{2}}\right)
}:=\left( \sum_{i_{1}=1}^{\infty }\left\Vert \left( x_{i_{1},i_{2}}\right)
_{i_{2=1}}^{\infty }\right\Vert _{\ell _{p_{2}}}^{p_{1}}\right) ^{\frac{1}{%
p_{1}}}<+\infty
\end{equation*}%
and by $\ell _{\infty }\left( \ell _{p_{2}}\right) $ the Banach space of the
sequences $\mathbf{x}=\left( x_{i_{1},i_{2}}\right) _{i_{1},i_{2}=1}^{\infty
}$ such that%
\begin{equation*}
\left\Vert \mathbf{x}\right\Vert _{\ell _{\infty }\left( \ell
_{p_{2}}\right) }:=\sup_{i_{1}}\left\Vert \left( x_{i_{1},i_{2}}\right)
_{i_{2=1}}^{\infty }\right\Vert _{\ell _{p_{2}}}<\infty ,
\end{equation*}%
Inductively, for $\mathbf{p}=(p_{1},\dots ,p_{m})\in \lbrack 1,+\infty ]^{m}$%
, we can define the Banach space $\ell _{\mathbf{p}}$ by
\begin{equation*}
\ell _{\mathbf{p}}:=\ell _{p_{1}}\left( \ell _{p_{2}}\left( \cdots \left(
\ell _{p_{m}}\right) \cdots \right) \right) .
\end{equation*}%
Namely, a vector $\mathbf{x}=\left( x_{i_{1},...,i_{m}}\right)
_{i_{1},...,i_{m}=1}^{\infty }\in \ell _{\mathbf{p}}$ if, and only if,
\begin{equation*}
\left( \sum_{i_{1}=1}^{\infty }\left( \sum_{i_{2}=1}^{\infty }\left( \dots
\left( \sum_{i_{m-1}=1}^{\infty }\left( \sum_{i_{m}=1}^{\infty }\left\vert
x_{i_{1},...,i_{m}}\right\vert ^{p_{m}}\right) ^{\frac{p_{m-1}}{p_{m}}%
}\right) ^{\frac{p_{m-2}}{p_{m-1}}}\dots \right) ^{\frac{p_{2}}{p_{3}}%
}\right) ^{\frac{p_{1}}{p_{2}}}\right) ^{\frac{1}{p_{1}}}<+\infty ,
\end{equation*}%
(the usual modification is required if some $p_{j}=\infty $).

The next result is based on ideas borrowed from \cite[Theorem 3.2]{kato}. We
use standard notation and notions from interpolation theory, as presented
e.g. in \cite{Bergh}.

\begin{theorem}
Let $\mathbf{p}=(p_{1},\dots ,p_{m})\in (1,+\infty )^{m}$, and let $t:=\min
\left\{ p_{1},...,p_{m},p_{1}^{\ast },...,p_{m}^{\ast }\right\} $. Then,%
\begin{equation}
\left\Vert R_{n}:\ell _{t}^{n}\left( \ell _{\mathbf{p}}\right) \rightarrow
\ell _{s}^{2^{n}}\left( \ell _{\mathbf{p}}\right) \right\Vert =2^{\frac{n}{s}%
}  \label{star}
\end{equation}%
for any $s$, $1\leq s\leq t^{\ast }$ and all $n\in \mathbb{N}$. In other
words, $\ell _{\mathbf{p}}$ is of type $t$ and $T_{t,s}\left( \ell _{\mathbf{%
p}}\right) =1$, for all $1\leq s\leq t^{\ast }$.
\end{theorem}

\begin{proof}
It is enough to show (\ref{star}) for $s=t^{\ast }$. By Lemma \ref{lemmajapa}
we know that for all $n\in \mathbb{N}$,
\begin{equation}
\left\Vert R_{n}:\ell _{2}^{n}\left( \ell _{\mathbf{2}}\right) \rightarrow
\ell _{2}^{2^{n}}\left( \ell _{\mathbf{2}}\right) \right\Vert =2^{\frac{n}{2}%
},  \label{2}
\end{equation}%
because $\ell _{\mathbf{2}}:=\ell _{2}\left( \ell _{2}\left( \cdots \left(
\ell _{2}\right) \cdots \right) \right) $ is a Hilbert space.

Suppose that $p_{1},...,p_{m}$ are not all $2$. As a first step, let us show
\begin{equation}
\left\Vert R_{n}:\ell _{t}^{n}\left( \ell _{\mathbf{p}}\right) \rightarrow
\ell _{t^{\ast }}^{2^{n}}\left( \ell _{\mathbf{p}}\right) \right\Vert \leq
2^{\frac{n}{t^{\ast }}}.  \label{auxi}
\end{equation}

If $t=p_{k}$ for some $k\in \left\{ 1,...,m\right\} $ (obviously $p_{k}<2$),
put $\theta =\frac{2}{p_{k}^{\ast }}\in \left( 0,1\right) $ and
\begin{equation*}
\frac{1}{p_{i}^{0}}:=\frac{1/p_{i}-1/p_{k}^{\ast }}{1/p_{k}-1/p_{k}^{\ast }}
\end{equation*}%
for all $i\in \left\{ 1,...,m\right\} $, $i\not=k$.

Then, since
\begin{equation*}
\frac{1-\theta }{1}+\frac{\theta }{2}=\frac{1}{p_{k}},
\end{equation*}%
\begin{equation*}
\frac{1-\theta }{p_{i}^{0}}+\frac{\theta }{2}=\frac{1}{p_{i}},
\end{equation*}%
for all $i\in \left\{ 1,...,m\right\} $, $i\not=k$, we have by \cite[Theorem
5.1.1 and Theorem 5.1.2]{Bergh}
\begin{equation*}
\left( \ell_{\mathbf{p}_{0}},\ell_{\mathbf{2}}\right) _{\left[ \theta \right]
}=\ell_{\mathbf{p}}
\end{equation*}%
with equal norms, where $\mathbf{p}_{0}=(p_{1}^{0},\dots
,p_{k-1}^{0},1,p_{k+1}^{0},\dots ,p_{m}^{0})$.

With the same notation, since
\begin{equation*}
\frac{1-\theta }{\infty }+\frac{\theta }{2}=\frac{1}{p_{k}^{\ast }},
\end{equation*}%
we have by \cite[Theorem 4.2.1, Theorem 5.1.1 and Theorem 5.1.2]{Bergh}%
\begin{equation}
\left[ \ell _{1}^{n}\left( \ell _{\mathbf{p}_{0}}\right) ,\ell
_{2}^{n}\left( \ell _{\mathbf{2}}\right) \right] _{\left[ \theta \right]
}=\ell _{p_{k}}^{n}\left( \ell _{\mathbf{p}}\right)  \label{inter}
\end{equation}%
and%
\begin{equation}
\left[ \ell _{\infty }^{2^{n}}\left( \ell _{\mathbf{p}_{0}}\right) ,\ell
_{2}^{2^{n}}\left( \ell _{\mathbf{2}}\right) \right] _{\left[ \theta \right]
}=\ell _{p_{k}^{\ast }}^{2^{n}}\left( \ell _{\mathbf{p}}\right)  \label{pol}
\end{equation}%
with equal norms.

Computing,
\begin{equation}
\left\Vert R_{n}:\ell _{1}^{n}\left( \ell _{\mathbf{p}_{0}}\right)
\rightarrow \ell _{\infty }^{2^{n}}\left( \ell _{\mathbf{p}_{0}}\right)
\right\Vert =1  \label{1}
\end{equation}%
and interpolating (\ref{2}) and (\ref{1}) (by using \ref{inter}, and \ref%
{pol}) we have%
\begin{equation*}
\left\Vert R_{n}:\ell _{p_{k}}^{n}\left( \ell _{\mathbf{p}}\right)
\rightarrow \ell _{p_{k}^{\ast }}^{2^{n}}\left( \ell _{\mathbf{p}}\right)
\right\Vert \leq \left( 2^{\frac{n}{2}}\right) ^{\theta }=2^{\frac{n}{%
p_{k}^{\ast }}},
\end{equation*}%
and then (\ref{auxi}) is true if $t=p_{k}$ for some $k\in \left\{
1,...,m\right\} $.

If $t=p_{k}^{\ast }$ for some $k\in \left\{ 1,...,m\right\} $, (obviously $%
p_{k}>2$), put $\theta =\frac{2}{p_{k}}\in \left( 0,1\right) $ and
\begin{equation*}
\frac{1}{p_{i}^{1}}:=\frac{1/p_{i}-1/p_{k}}{1/p_{k}^{\ast }-1/p_{k}}
\end{equation*}%
for all $i\in \left\{ 1,...,m\right\} $, $i\not=k$.

Then, since
\begin{equation*}
\frac{1-\theta }{\infty }+\frac{\theta }{2}=\frac{1}{p_{k}},
\end{equation*}%
\begin{equation*}
\frac{1-\theta }{p_{i}^{1}}+\frac{\theta }{2}=\frac{1}{p_{i}},
\end{equation*}%
for all $i\in \left\{ 1,...,m\right\} $, $i\not=k$, we have%
\begin{equation*}
\left( \ell _{\mathbf{p}_{1}},\ell _{\mathbf{2}}\right) _{\left[ \theta %
\right] }=\ell _{\mathbf{p}}
\end{equation*}%
with equal norms, where $\mathbf{p}_{1}=(p_{1}^{1},\dots ,p_{k-1}^{1},\infty
,p_{k+1}^{1},\dots ,p_{m}^{1})$.

Keeping the notation, since
\begin{equation*}
\frac{1-\theta }{1}+\frac{\theta }{2}=\frac{1}{p_{k}^{\ast }},
\end{equation*}%
we have by \cite[Theorem 4.2.1, Theorem 5.1.1 and Theorem 5.1.2]{Bergh}%
\begin{equation}
\left[ \ell _{1}^{n}\left( \ell _{\mathbf{p}_{1}}\right) ,\ell
_{2}^{n}\left( \ell _{\mathbf{2}}\right) \right] _{\left[ \theta \right]
}=\ell _{p_{k}}^{n}\left( \ell _{\mathbf{p}}\right)  \label{inte}
\end{equation}%
and%
\begin{equation}
\left[ \ell _{\infty }^{2^{n}}\left( \ell _{\mathbf{p}_{1}}\right) ,\ell
_{2}^{2^{n}}\left( \ell _{\mathbf{2}}\right) \right] _{\left[ \theta \right]
}=\ell _{p_{k}^{\ast }}^{2^{n}}\left( \ell _{\mathbf{p}}\right)  \label{po}
\end{equation}%
with equal norms, and
\begin{equation}
\left\Vert R_{n}:\ell _{1}^{n}\left( \ell _{\mathbf{p}_{1}}\right)
\rightarrow \ell _{\infty }^{2^{n}}\left( \ell _{\mathbf{p}_{1}}\right)
\right\Vert =1.  \label{3}
\end{equation}
By using (\ref{inte}), (\ref{po}), and interpolating (\ref{2}) and (\ref{3})
we have%
\begin{equation*}
\left\Vert R_{n}:\ell _{p_{k}^{\ast }}^{n}\left( \ell _{\mathbf{p}}\right)
\rightarrow \ell _{p_{k}}^{2^{n}}\left( \ell _{\mathbf{p}}\right)
\right\Vert \leq \left( 2^{\frac{n}{2}}\right) ^{\theta }=2^{\frac{n}{p_{k}}%
}.
\end{equation*}%
Therefore, the inequality (\ref{auxi}) is true.

To show%
\begin{equation*}
\left\Vert R_{n}:\ell _{t}^{n}\left( \ell _{\mathbf{p}}\right) \rightarrow
\ell _{t^{\ast }}^{2^{n}}\left( \ell _{\mathbf{p}}\right) \right\Vert =2^{%
\frac{n}{t^{\ast }}},
\end{equation*}%
it is enough to see that the equality is attained with $(\mathbf{x},\mathbf{0%
},...,\mathbf{0})\in \ell _{t}^{n}\left( \ell _{\mathbf{p}}\right) $, $\
\mathbf{0}\not=\mathbf{x}=\left( x_{i_{1},...,i_{m}}\right)
_{i_{1},...,i_{m}=1}^{\infty }\in \ell _{\mathbf{p}}$, and the proof is done.
\end{proof}

The following result is a direct consequence of the above theorem and
Proposition \ref{propojapa}, using duality and the reflexivity of $\ell_{%
\mathbf{p}}$:

\begin{corollary}
\label{cotlp}Let $\mathbf{p}=(p_{1},\dots ,p_{m})\in (1,+\infty )^{m}$, and
let $t:=\min \left\{ p_{1},...,p_{m},p_{1}^{\ast },...,p_{m}^{\ast }\right\}
$. Then, for any $s$ with $t\leq s<\infty $, we have%
\begin{equation*}
\left\Vert t_{R_{n}}:\ell _{s}^{2^{n}}\left( \ell _{\mathbf{p}}\right)
\rightarrow \ell _{t^{\ast }}^{n}\left( \ell _{\mathbf{p}}\right)
\right\Vert =2^{\frac{n}{s^{\ast }}}
\end{equation*}%
for all $n\in \mathbb{N}$. Hence, $\ell _{\mathbf{p}}$ is of cotype $t^{\ast
}$ and
\begin{equation}
C_{t^{\ast },s}\left( \ell _{\mathbf{p}}\right) =1,  \label{888999}
\end{equation}%
for all $t\leq s<\infty $.
\end{corollary}

\smallskip

\begin{remark}
The above corollary was proved by using \emph{complex interpolation}, for
the case of complex scalars but from the very definition of cotype it is
obvious that (\ref{888999}) also holds for real Banach spaces.
\end{remark}

\medskip

\begin{remark}
(\textbf{Optimal constants for Theorems \ref{661} and \ref{t1multi}}) (1)
From the previous results we conclude that in Theorem \ref{661}, when $%
Y=\ell _{r}$ (over the real or complex field) with $r\in \lbrack 2,\infty )$%
, if (b) is true, then the optimal constant in the inequality (a) satisfies $%
C_{p_{1},...,p_{m}}^{\ell _{r}}=1.$ In fact, the case $m=1$ is immediate.
For $m=2$, note that $\ell _{\lambda _{1,r}^{p_{2}}}\left( \ell _{r}\right) $
has cotype $\lambda _{1,r}^{p_{2}}:=R>r\geq 2$ (by Corollary \ref{cotlp})
with $C_{R}(\ell _{R}\left( \ell _{r}\right) )=1$, and thus, following the
proof of Theorem \ref{661}, $C_{p_{1},p_{2}}^{\ell _{r}}\leq C_{R}(\ell
_{R}\left( \ell _{r}\right) )C_{r}\left( \ell _{r}\right) =1$. For the case $%
m=3$, following the proof of Theorem \ref{661}, we have%
\begin{equation*}
C_{p_{1},p_{2},p_{3}}^{\ell _{r}}\leq C_{\lambda _{2,r}^{p_{2},p_{3}}}\left(
\ell _{\lambda _{2,r}^{p_{2},p_{3}}}\left( \ell _{\lambda
_{1,r}^{p_{3}}}\left( \ell _{r}\right) \right) \right) C_{\lambda
_{1,r}^{p_{3}}}\left( \ell _{\lambda _{1,r}^{p_{3}}}\left( \ell _{r}\right)
\right) C_{r}\left( \ell _{r}\right) =1
\end{equation*}%
by Corollary \ref{cotlp} and the proof follows inductively.

(2) If (b) is true in Theorem \ref{t1multi} the optimal constat $%
C_{p_{1},...,p_{m}}$ in (a) is $1$, because the $\left( m-1\right) $-linear
operator used in the argument of the proof of (b)$\Rightarrow $(a) has range
$\ell _{\delta _{1}^{p_{m}}}=\ell _{\left( p_{m}\right) ^{\ast }}$ and $%
\left( p_{m}\right) ^{\ast }\geq 2$. By the first item of this remark, we
know that
\begin{equation*}
C_{p_{1},...,p_{m-1}}^{\ell _{\delta _{1}^{p_{m}}}}=1.
\end{equation*}
\end{remark}

\medskip

\begin{remark}
All the above results can be translated to the setting of multiple summing
operators (for recent results on multiple summing operators we refer to \cite%
{anss, popa2, bla, popa1, rueda} and references therein). In fact, we just
need to consider the more general concept of multiple summing operators
introduced in \cite{pilar} and recall how to translate coincidence
situations like the Bohnenblust--Hille inequality and Hardy--Littlewood
inequalities to multiple summing operators (see, for instance, \cite{dimant}
and \cite[Corollary 3.20]{pv}).
\end{remark}

\textbf{Acknowledgement.} Much of this work was done while the second and
fourth authors were visitors to the Department of Mathematics of Kent State
University. These authors offer their thanks to this Department, and also to
Manuel Maestre, for making their visit productive.


\begin{thebibliography}{99}
\bibitem{albiac} F. Albiac, N. Kalton, Topics in Banach Space Theory,
Graduate Texts in Mathematics 233, Springer-Verlag 2005.

\bibitem{pilar} N. Albuquerque, G. Ara\'{u}jo, D. N\'{u}\~{n}ez-Alarc\'{o}n,
D. Pellegrino, P. Rueda, Bohnenblust--Hille and Hardy--Littlewood
inequalities by blocks, arXiv:1409.6769v6 [math.FA] 26 Oct 2015.

\bibitem{abps2} N. Albuquerque, F. Bayart, D. Pellegrino, J. B. Seoane-Sep%
\'{u}lveda, Optimal Hardy-Littlewood type inequalities for polynomials and
multilinear operators, Israel J. Math 211 (2016), 197--220.

\bibitem{aaa} N. Albuquerque, F. Bayart, D. Pellegrino, J. B. Seoane-Sep\'{u}%
lveda, Sharp generalizations of the multilinear Bohnenblust-Hille
inequality. J. Funct. Anal. 266 (2014), no. 6, 3726--3740.

\bibitem{anss} N. Albuquerque, D. N\'{u}\~{n}ez-Alarc\'{o}n, J. Santos, D.
M. Serrano-Rodr\'{\i}guez, Absolutely summing linear operators via
interpolation, J. Funct. Anal., \textbf{269} (2015), 1636-1651.

\bibitem{popa2} G. Badea, D. Popa, Swartz type results for nuclear and
multiple 1-summing bilinear operators on $c_{0}\left( X\right) \times
c_{0}\left( Y\right) $. Positivity 19 (2015), no. 3, 475--487.

\bibitem{Bergh} J. Bergh, J. L\"{o}fstr\"{o}m, Interpolation spaces,
Springer-Verlag, Berlin-Heidelberg-New York, 1976.

\bibitem{bla} O. Blasco, G. Botelho, D. Pellegrino, P. Rueda, Summability of
multilinear mappings: Littlewood, Orlicz and beyond. Monatsh. Math. 163
(2011), no. 2, 131--147.

\bibitem{carando} D. Carando, V. Dimant, S. Muro, D. Pinasco, An integral
formula for multiple summing norms of operators. Linear Algebra Appl. 478
(2015), 274--293.

\bibitem{was} W. Cavalcante, D. N\'{u}\~{n}ez-Alarc\'{o}n, Remarks on an
Inequality of Hardy and Littlewood. To appear in Quaest. Mathematicae.

\bibitem{Di} J. Diestel, H. Jarchow, A. Tonge, Absolutely summing operators,
Cambridge Univ. Press, Cambridge, 1995.

\bibitem{dimant} V. Dimant, P. Sevilla--Peris, Summation of coefficients of
polynomials on $\ell _{p}$ spaces, to appear in Publ. Mathematiques.

\bibitem{hardy} G. Hardy, J. E. Littlewood, Bilinear forms bounded in space $%
[p,q]$, Quart. J. Math. \textbf{5} (1934), 241--254.

\bibitem{kato} M. Kato, K. Miyazaki, Y. Takahashi, Type, cotype constants
for~$L_{p}\left( L_{q}\right) $\ norms of the Rademacher matrices and
interpolation, Nihonkai Math. J. \textbf{6 }(1995), 81-95.

\bibitem{pisier} B. Maurey, G. Pisier, S\'{e}ries de variables al\'{e}%
atories vectorielles ind\'{e}pendants et propri\'{e}t\'{e}s g\'{e}om\'{e}%
triques des \'{e}spaces de Banach, Studia Math., \textbf{58 }(1976)\textbf{,
}45--90.

\bibitem{tonge} B. Osikiewicz, A. Tonge, An interpolation approach to
Hardy--Littlewood inequalities for norms of operators on sequence spaces,
Linear Algebra Appl. \textbf{331} (2001), 1--9.

\bibitem{pv} D. P\'{e}rez-Garc\'{\i}a, I. Villanueva, Multiple summing
operators on $C(K)$ spaces, Ark. Mat. \textbf{42} (2004), no. 1, 153--171.

\bibitem{popa1} D. Popa, Multiple summing operators on lp spaces. Studia
Math. 225 (2014), no. 1, 9--28.

\bibitem{pra} T. Praciano-Pereira, On bounded multilinear forms on a class
of $\ell _{p}$ spaces. J. Math. Anal. Appl. \textbf{81} (1981), 561--568.

\bibitem{rueda} P. Rueda, E. A. S\'{a}nchez-P\'{e}rez, Factorization of
p-dominated polynomials through Lp-spaces. Michigan Math. J. 63 (2014), no.
2, 345--353.
\end{thebibliography}
\end{document}